\documentclass[11pt]{article}
\usepackage[english]{babel}
\usepackage[margin=1in]{geometry}
\usepackage{amssymb,amsmath,amsthm,mathtools}
\usepackage{microtype}
\usepackage[round,authoryear]{natbib}
\usepackage{booktabs}
\usepackage[T1]{fontenc}
\usepackage{times}
\usepackage[colorlinks=true,urlcolor=blue,linkcolor=blue,citecolor=red]{hyperref}
\newtheorem{theorem}{Theorem}[section]
\newtheorem{lemma}[theorem]{Lemma}
\newtheorem{proposition}[theorem]{Proposition}
\newtheorem{corollary}[theorem]{Corollary}
\numberwithin{equation}{section}
\newcommand{\R}{\mathbb R}
\newcommand{\E}{\mathbb E}
\newcommand{\PP}{\mathbb P}
\newcommand{\F}{\mathcal F}
\newcommand{\one}{\mathbf1}
\newcommand{\dd}{\,\mathrm d}
\newcommand{\convp}{\xrightarrow{p}}
\newcommand{\convd}{\Rightarrow}
\newcommand{\rise}[2]{#1^{\overline{#2}}}
\DeclareMathOperator{\Var}{Var}
\DeclareMathOperator{\Cov}{Cov}
\DeclareMathOperator{\Lip}{Lip}
\title{Asymptotic Normality of the Hill Estimator\\
in Linear Preferential Attachment}
\author{Tiandong Wang \thanks{Shanghai Center for Mathematical Sciences, Fudan University, Shanghai, China.
    Email: \texttt{td\_wang@fudan.edu.cn}.}}
\date{}
\hypersetup{pdftitle={Asymptotic Normality of the Hill Estimator in Linear Preferential Attachment},
 pdfauthor={Tiandong Wang},
 pdfkeywords={preferential attachment, Hill estimator, asymptotic normality, extreme degrees}}
\begin{document}
\maketitle
\begin{abstract}
We prove asymptotic normality of the Hill estimator based on the
largest $k_n$ degrees in a homogeneous linear preferential attachment
graph with $n$ vertices. Each new vertex sends a fixed number
$m\ge1$ of edges sequentially to older vertices, with attachment
weight $d+\delta$ and fixed $\delta>-m$. Let
$\alpha=m/(2m+\delta)$. For every deterministic sequence
$k_n\to\infty$ satisfying $\sqrt{k_n}(k_n/n)^\alpha\to0$, the
estimator, centered at $\alpha$ and multiplied by $\sqrt{k_n}$,
converges to a centered normal law with an explicit variance strictly smaller than $\alpha^2$.
The reduction relative to the Hill variance for independent observations
separates into contributions from the different degree laws at different
birth times and from competition for new edges. The contribution from
competition is $1/m$ times that from different birth times.
Representations using Gamma and beta functions give a consistent
standard error and a studentized normal limit. The proof uses martingales
indexed by edge attachments and comparisons with independent
pure-birth processes, while retaining the integer degrees and their ties.
\end{abstract}

\section{Introduction}\label{sec:intro}
Preferential attachment produces power-law degree distributions by
giving vertices with larger degrees a greater chance of receiving
new edges \citep{barabasi1999emergence}. For linear attachment with a
fixed number $m\ge1$ of edges per new vertex and weight $d+\delta$,
where $\delta>-m$, the limiting degree distribution has tail exponent
$2+\delta/m$ \citep{bollobas2001degree,hofstad2017}. The Hill
estimator \citep{hill1975simple} estimates its reciprocal,
$\alpha=m/(2m+\delta)$, from the largest degrees. To quantify the
uncertainty of this estimate, one must account for the dependence
between degrees and their different distributions across birth times.

\citet{wang2019consistency} proved consistency of the Hill estimator
in a linear preferential attachment model with one edge per new
vertex. Their analysis combines birth-process embeddings, degree
concentration and convergence of the tail empirical measure.
\citet{wang2020degree} established related consistency results for
the in- and out-degrees of a directed model. These results identify
the probability limits of the estimators. A normal limit additionally
requires control of the fluctuations of the empirical measure and
of the order statistic used as the threshold.

\citet{resnick2016normality} and \citet{baldassarri2022normality}
proved joint central limit theorems for degree counts, with the latter
covering any fixed number of edges per arriving vertex. Their
convergence is in the product topology, so it controls every fixed
collection of degree values. The Hill estimator uses a growing number
$k_n$ of upper degrees in a graph with $n$ vertices, with
$k_n/n\to0$. Its threshold diverges with $n$. Applying a fixed-degree
central limit theorem at that threshold would require additional
uniform estimates, including control of the logarithmic summands.

Using the attachment likelihood, \citet{gao2017normality} proved
asymptotic normality and efficiency
of maximum-likelihood estimation for the affine parameter. When the
initial degree of each vertex is the same fixed integer, the final
degree sequence is sufficient for their likelihood, and the estimator
has rate $\sqrt n$. Hill estimation uses only the upper degrees and
has a different normalization. Its limiting distribution therefore
requires a separate analysis even when only one graph is observed.

The classical Hill central limit theorem for independent observations
has variance $\alpha^2$ under suitable conditions on the tail and
the number of upper order statistics
\citep{haeusler1985asymptotic,dehaan2006extreme}. Extensions to
stationary dependent observations require conditions on their joint
tails \citep{hsing1991tail,resnick1998tail}. For the Norros--Reittu
and Chung--Lu graph models, \citet{cirkovic2024tail} proved normal
limits by comparing the largest degrees with independent vertex
weights. Their approximation preserves the iid Hill variance.
In preferential attachment, each new edge increases exactly one
older vertex's degree, and the resulting covariance contributes to
the limiting variance.

We prove a central limit theorem for the Hill estimator in the
sequential homogeneous model for every fixed $m\ge1$ and $\delta>-m$.
Under~\eqref{eq:k-rate}, Theorem~\ref{thm:clt} identifies a positive
limiting variance strictly smaller than $\alpha^2$, and
Corollary~\ref{cor:studentized} gives a studentized normal limit.
The reduction separates into contributions from different degree
laws at different birth times and from competition for edges.
The first contribution parallels the
variance reduction for independent heterogeneous observations in
\citet{einmahl2023extreme}. The second records the dependence created
by the attachment rule. This calculation gives a standard error for
Hill estimation that incorporates both features of the graph.

The proof follows the terminal degree statistic through individual
edge attachments. The predictable variance of the resulting
martingale retains the covariance between competing vertices.
Comparisons with independent pure-birth processes identify its limit,
and moment bounds control vertices born very early or very late.
We then approximate the discontinuous logarithmic summand and invert
the empirical tail at the random threshold. Bounds on the expected
degree profile remove the mean bias under~\eqref{eq:k-rate} and make
the effect of the integer lattice negligible at scale $k_n^{-1/2}$.

Section~\ref{sec:model} defines the graph and states the results.
Section~\ref{sec:clt} contains their proofs, and
Section~\ref{sec:discussion} describes the role of the degree lattice
and the scope of the theorem. Appendix~\ref{sec:tools} provides the
degree-process estimates, and Appendix~\ref{sec:examples} evaluates
the variance and the stationary-profile bias in special cases.

\section{Model and results}\label{sec:model}
\subsection{Sequential attachment}
Fix an integer $m\ge1$ and an offset $\delta>-m$. At time $2$, two vertices are joined by $m$
parallel edges. For $t\ge3$, vertex $t$ sends $m$ edges to vertices in
$\{1,\ldots,t-1\}$. Let $D_v(t,i)$ be the degree of an older vertex
$v<t$ after the first $i$ of these edges, with
$D_v(t,0)=D_v(t-1)$ and $D_v(t,m)=D_v(t)$. The new vertex becomes an
eligible target at the next step, with degree $D_t(t)=m$.
Write $\F_{t,i}$ for the history through edge $i$, and let $J_{t,i}$
be its older endpoint. Given $\F_{t,i-1}$,
\begin{equation}\label{eq:attachment}
 \PP(J_{t,i}=v\mid\F_{t,i-1})
 =\frac{D_v(t,i-1)+\delta}{W_{t,i}},\qquad
 W_{t,i}=(2m+\delta)(t-1)-2m+i-1.
\end{equation}
The denominator is deterministic because the older vertices have
total degree $2m(t-2)+i-1$ before edge $i$. Multiple edges are
allowed and loops are excluded. Degrees are updated after each edge,
and the arriving vertex cannot be selected by its own edges.
All results refer to this attachment convention.

Write
\begin{equation}\label{eq:parameters}
 \begin{aligned}
 p&=\frac{2m+\delta}{m}>1,\qquad \alpha=\frac1p\in(0,1),\\
 r&=m+\delta=m(p-1)>0,\qquad
 \beta^p=\frac{\Gamma(r+p)}{\Gamma(r)}.
 \end{aligned}
\end{equation}
Throughout $m$ and $\delta$ are fixed as $n\to\infty$.
Constants denoted by $C$ may depend on them and on fixed arguments of
the statement in which they occur.

Let $D_{(1)}(n)\ge\cdots\ge D_{(n)}(n)$ be the ordered degrees. For
$1\le k<n$, the Hill estimator \citep{hill1975simple} is
\begin{equation}\label{eq:hill}
 \widehat\alpha_n(k)=\frac1k\sum_{j=1}^k
             \log\frac{D_{(j)}(n)}{D_{(k+1)}(n)}.
\end{equation}
A degree equal to $D_{(k+1)}(n)$ contributes zero to~\eqref{eq:hill},
so ties need no separate convention. We take $k=k_n\to\infty$.

\subsection{The central limit theorem}
For $a>0$, let $G_a$ have the Gamma distribution of shape $a$ and scale
one, and let
\[
 \Gamma(a,x)=\int_x^\infty y^{a-1}e^{-y}\dd y
\]
be the upper incomplete Gamma function. Define
\begin{equation}\label{eq:variance-function}
 \begin{split}
 \Lambda_{m,\delta}
 &=\frac{1}{p\Gamma(r)\Gamma(r+p)}
          \int_0^\infty x^{p-1}\Gamma(r,x)^2\dd x,\\
 V_m(\delta)&=\alpha^2-(1+1/m)\Lambda_{m,\delta}.
 \end{split}
\end{equation}
The integral is finite: $\Gamma(r,x)\le2^r\Gamma(r)e^{-x/2}$ for
$x\ge0$. The following theorem gives the limiting variance and the
rate needed for centering at $\alpha$.

\begin{theorem}\label{thm:clt}
For the graph defined by~\eqref{eq:attachment}, let $m\ge1$ and $\delta>-m$.
Let $k_n$ be deterministic integers with $1\le k_n<n$. If
$k_n\to\infty$ and
\begin{equation}\label{eq:k-rate}
 \sqrt{k_n}(k_n/n)^\alpha\longrightarrow0,
\end{equation}
then $\sqrt{k_n}\{\widehat\alpha_n(k_n)-\alpha\}\convd N(0,V_m(\delta))$
in $\R$.
\end{theorem}

Condition~\eqref{eq:k-rate} is equivalent to
$k_n=o(n^{2m/(4m+\delta)})$ and implies $k_n/n\to0$.
It makes the error caused by the integer lattice of degrees negligible
at the scale $k_n^{-1/2}$, and no lower growth condition on $k_n$ is
needed. Section~\ref{sec:discussion} shows that~\eqref{eq:k-rate} is
also necessary for the local approximation of the random threshold
used in the proof. The next result compares $V_m(\delta)$ with
$\alpha^2$, the asymptotic variance of the Hill estimator for
independent observations with extreme-value index $\alpha$
\citep{haeusler1985asymptotic,dehaan2006extreme}.

\begin{corollary}\label{cor:variance}
For every $m\ge1$ and $\delta>-m$, $0<V_m(\delta)<\alpha^2$.
\end{corollary}

The reduction $\alpha^2-V_m(\delta)$ separates into
$\Lambda_{m,\delta}$, due to the different degree laws at different
birth times, and $\Lambda_{m,\delta}/m$, due to competition for each
new edge. Lemma~\ref{lem:variance-value} derives this decomposition
and proves that $V_m$ is continuous on $(-m,\infty)$.
The first term parallels the reduction for independent heterogeneous
observations studied by \citet{einmahl2023extreme}.

The variance also has a convenient expression in terms of the
regularized incomplete beta function. For $a,b>0$ and $0\le x\le1$,
define
\[
 I_x(a,b)=\frac{\Gamma(a+b)}{\Gamma(a)\Gamma(b)}
               \int_0^x t^{a-1}(1-t)^{b-1}\dd t.
\]
Lemma~\ref{lem:variance-value} gives
\begin{equation}\label{eq:variance-beta}
 V_m(\delta)=\frac1{p^2}
       \left\{1-2(1+m^{-1})I_{1/2}(r+p,r)\right\}.
\end{equation}
This expression permits direct evaluation of an asymptotic standard
error. To define it on every sample, set
$\widehat\delta_n=m\{1/\widehat\alpha_n(k_n)-2\}$ when
$0<\widehat\alpha_n(k_n)<1$, and set $\widehat\delta_n=0$ otherwise.
Then $\widehat\delta_n>-m$ in either case. Put
$\widehat V_n=V_m(\widehat\delta_n)$.

\begin{corollary}\label{cor:studentized}
Under the assumptions of Theorem~\ref{thm:clt},
\[
 \frac{\sqrt{k_n}\{\widehat\alpha_n(k_n)-\alpha\}}
      {\sqrt{\widehat V_n}}\convd N(0,1)\quad\hbox{in }\R.
\]
If $z_q$ is the $q$-quantile of $N(0,1)$ and $0<\gamma<1$, the
interval
\[
 \left[\widehat\alpha_n(k_n)-z_{1-\gamma/2}\sqrt{\widehat V_n/k_n},
       \widehat\alpha_n(k_n)+z_{1-\gamma/2}\sqrt{\widehat V_n/k_n}\right]
\]
contains $\alpha$ with probability tending to $1-\gamma$.
\end{corollary}

\begin{proof}
Theorem~\ref{thm:clt} implies
$\widehat\alpha_n(k_n)\convp\alpha\in(0,1)$, so the default value of
$\widehat\delta_n$ is used with probability tending to zero and
$\widehat\delta_n\convp\delta$. Continuity and positivity of $V_m$
give $\widehat V_n\convp V_m(\delta)>0$. Slutsky's theorem proves
both assertions.
\end{proof}

\section{Proof of the central limit theorem}\label{sec:clt}
Write $k=k_n$. Except where~\eqref{eq:k-rate} is invoked, assume
only $k\to\infty$ and $k/n\to0$. We use products of the deterministic
edge denominators to normalize degree growth. An edge $e=(t,i)$
occurs at step $t$, and $W_e=W_{t,i}$. Integer endpoints refer to the
completion of all $m$ edges at that step. Products over edge times
follow chronological order. For $q\ge0$, define
\begin{equation}\label{eq:products}
 R_q(s,t)=\prod_{s<e\le t}\left(1+\frac q{W_e}\right),
 \qquad \Theta(s,t)=\log R_1(s,t),
 \qquad g(s,t)=(t/s)^\alpha.
\end{equation}
The definition of $R_q$ also allows endpoints inside a step.
For such endpoints, $s$ and $t$ in bounds involving $g$ denote the
corresponding step indices. Lemma~\ref{lem:clock} shows that
$R_q(s,t)/g(s,t)^q=1+O(s^{-1})$ for integer endpoints.
The seed vertices have birth time $b_1=b_2=2$, and $b_v=v$ for $v\ge3$.

For $x>0$ and $q\ge0$, write
$\rise{x}{q}=\Gamma(x+q)/\Gamma(x)$ for the rising factorial.
For a bounded function $f$, $\|f\|_\infty$ is its supremum norm and
$\Lip(f)$ is its Lipschitz constant when finite.

The empirical measure of large degrees is normalized by
$a_n=\beta(n/k)^\alpha$ so that its limiting tail at $x>0$ is
$x^{-p}$. Define
\begin{equation}\label{eq:empirical-measure}
 a_n=\beta(n/k)^\alpha,\qquad
 \widehat\nu_n(A)=\frac1k\sum_{v=1}^n\one\{D_v(n)/a_n\in A\},\qquad
 \nu(\dd x)=px^{-p-1}\dd x,
\end{equation}
for Borel sets $A\subset(0,\infty)$.
Write $\log_+x=\max(\log x,0)$. For a measurable function $f$, write
$\widehat\nu_nf=\int f\dd\widehat\nu_n$, $\nu f=\int f\dd\nu$, and
$S_n(f)=\sqrt{k}\{\widehat\nu_nf-\E\widehat\nu_nf\}$ whenever these
integrals exist. Define
\[
 h(x)=(\log x-\alpha)\one\{x>1\},\qquad
 \widehat x_n=D_{(k+1)}(n)/a_n.
\]
Integration of the empirical tail gives the exact decomposition
\begin{equation}\label{eq:hill-decomposition}
\begin{split}
 \sqrt{k}\{\widehat\alpha_n(k)-\alpha\}
 &=S_n(h)+\sqrt{k}\,\E\widehat\nu_nh\\
 &\quad+\sqrt{k}\left[
 \alpha\{\widehat\nu_n(1,\infty)-1\}
 -\int_1^{\widehat x_n}\widehat\nu_n(x,\infty)\frac{\dd x}{x}\right].
\end{split}
\end{equation}
Indeed, $k\widehat\alpha_n(k)=
\sum_v\log_+\{D_v(n)/D_{(k+1)}(n)\}$, including ties.
We first control the mean, then prove a centered empirical central
limit theorem, and finally handle the random threshold.

\subsection{Expected degree counts}
\label{sec:mean}

The expected numbers of vertices of each degree satisfy a linear
recursion at every edge attachment.  For $j\ge m$, put
\begin{equation}\label{eq:pi}
 \pi_j=p\beta^p\frac{\Gamma(j+\delta)}{\Gamma(j+\delta+p+1)}.
\end{equation}
Telescoping gives
\begin{equation}\label{eq:pi-tail}
 \sum_{j\ge d}\pi_j
 =\beta^p\frac{\Gamma(d+\delta)}{\Gamma(d+\delta+p)},
 \qquad d\ge m.
\end{equation}
In particular, $\sum_{j\ge m}\pi_j=1$,
$\pi_j\le Cj^{-p-1}$ and $\sum_{j\ge d}\pi_j\le Cd^{-p}$.
Write
\[
 \mathcal N_j(t)=\sum_{v\le t}\one\{D_v(t)=j\},\qquad
 \mu_j(t)=\E\mathcal N_j(t),\qquad
 e_j(t)=\mu_j(t)-t\pi_j.
\]
The next lemma controls the entire expected degree profile, including
the range of large degrees that enters the Hill estimator.

\begin{lemma}\label{lem:mean-profile}
For every $q\in[0,p)$ there is a constant $C_q$ such that
\[
 \sum_{j\ge m}\rise{(j+\delta)}{q}|e_j(t)|
 \le C_qt^{q\alpha},\qquad t\ge2.
\]
In particular,
\[
 \sum_{j\ge m}|e_j(t)|\le C,
 \qquad
 \sum_{j\ge m}j|e_j(t)|\le Ct^\alpha.
\]
\end{lemma}

\begin{proof}
Set $\pi_j=0$ for $j<m$ and let $\mathbf e_m$ denote the unit mass at
$m$.  For sequences supported on $\{m,m+1,\ldots\}$, define
\[
 (Af)_j=(j-1+\delta)f_{j-1}-(j+\delta)f_j.
\]
Equation~\eqref{eq:pi} gives the identity
\begin{equation}\label{eq:mean-stationary}
 \alpha A\pi=\pi-\mathbf e_m.
\end{equation}
Immediately before the $i$th edge at step $t$, the total weight of the
old vertices is
$W_{t,i}=mp(t-1)-2m+i-1$.  The expectation of their degree-count vector
is therefore multiplied by $I+A/W_{t,i}$ at that edge.  After the
$m$th edge, the new vertex contributes one vertex of degree $m$.
Consequently,
\begin{equation}\label{eq:mean-count-recursion}
 \mu(t)=T_t\mu(t-1)+\mathbf e_m,
 \qquad
 T_t=\prod_{i=1}^m\left(I+\frac A{W_{t,i}}\right).
\end{equation}
All factors are polynomials in $A$, so they commute.

A constant correction makes the defect in
\eqref{eq:mean-count-recursion} summable.  Put
\[
 c_0=\frac{3m+1}{2mp},\qquad d_0=\frac{m-1}{2mp},\qquad
 v(t)=(t-c_0)\pi-d_0A\pi.
\]
Write $T_t=I+\sum_{\ell=1}^m s_\ell(t)A^\ell$, where $s_\ell(t)$
is the $\ell$th elementary symmetric polynomial in
$W_{t,1}^{-1},\ldots,W_{t,m}^{-1}$.  Direct expansion gives
\begin{align*}
 s_1(t)&=\frac\alpha t+
 \frac{\alpha+(3m+1)/(2mp^2)}{t^2}+O(t^{-3}),\\
 s_2(t)&=\frac{m-1}{2mp^2t^2}+O(t^{-3}),
 \qquad s_\ell(t)=O(t^{-\ell})\quad(\ell\ge3),
\end{align*}
where $s_2=0$ when $m=1$.  Put
$R(t)=T_tv(t-1)+\mathbf e_m-v(t)$.  Its constant term vanishes by
\eqref{eq:mean-stationary}.  The coefficients of $t^{-1}A\pi$ and
$t^{-1}A^2\pi$, respectively, are
\[
 \alpha+\frac{3m+1}{2mp^2}-\alpha(1+c_0)=0,
 \qquad
 \frac{m-1}{2mp^2}-d_0\alpha=0.
\]
Repeated use of \eqref{eq:mean-stationary} expresses every fixed
$A^\ell\pi$ as a constant multiple of $\pi$ plus a finitely supported
sequence.  Since $q<p$, these sequences have finite norm with weight
$\rise{(j+\delta)}{q}$.  These cancellations imply
\begin{equation}\label{eq:mean-profile-defect}
 \sum_{j\ge m}\rise{(j+\delta)}{q}|R_j(t)|\le C_qt^{-2}.
\end{equation}
Finitely many initial steps can be included by increasing $C_q$.

We next restrict the recursion to degrees that can occur at time $t$,
because $I+A/W_{t,i}$ need not have nonnegative entries on
arbitrarily large states.  The maximal degree at time $t$ is
$\bar d_t=m(t-1)$. Define
$v^{[t]}_j(t)=v_j(t)\one\{j\le \bar d_t\}$ and
$b(t)=\mu(t)-v^{[t]}(t)$.  Then
\[
 b(t)=T_tb(t-1)+\rho(t),\qquad
 \rho(t)=T_tv^{[t-1]}(t-1)+\mathbf e_m-v^{[t]}(t),
\]
and both $b(t)$ and $\rho(t)$ vanish above $\bar d_t$.  On $j\le \bar d_t$,
the difference between $\rho(t)$ and $R(t)$ involves only the omitted
source coordinates $\bar d_{t-1}<j\le \bar d_t$.  There are $m$ such
coordinates.  The corresponding entries of $T_t$ are bounded in
absolute value, and $v_j(t-1)=O(t^{-p})$ in this range, by
\eqref{eq:mean-stationary} and the bound on $\pi_j$.  Hence
\begin{equation}\label{eq:mean-profile-boundary}
 \sum_{j\ge m}\rise{(j+\delta)}{q}|\rho_j(t)|
 \le C_q\{t^{-2}+t^{q-p}\}.
\end{equation}

Starting from the support of $b(t-1)$, the degree before edge $i$ is
at most $m(t-2)+i-1$.  The coefficients of that edge transition are
nonnegative because
\[
 W_{t,i}-\{m(t-2)+i-1+\delta\}=(m+\delta)(t-2)>0.
\]
Moreover,
$\rise{(x+1)}{q}-\rise{x}{q}=q\rise{x}{q}/x$.
If $E_q(t)=\sum_j\rise{(j+\delta)}{q}|b_j(t)|$, positivity and
\eqref{eq:mean-profile-boundary} therefore give
\[
 E_q(t)\le
 \prod_{i=1}^m\left(1+\frac q{W_{t,i}}\right)E_q(t-1)
 +C_q\{t^{-2}+t^{q-p}\}.
\]
The product of these multipliers over steps $s+1,\ldots,t$ is at
most $C_q(t/s)^{q\alpha}$, since
$\sum_{i=1}^mW_{u,i}^{-1}=\alpha/u+O(u^{-2})$.
Iteration yields
\[
 E_q(t)\le C_qt^{q\alpha}
 \left(1+\sum_{s\ge3}s^{-q\alpha-2}
          +\sum_{s\ge3}s^{q-p-q\alpha}\right)
 \le C_qt^{q\alpha}.
\]
The second sum converges because $q-p-q\alpha<-1$ is equivalent to
$q<p$.  Finally,
\[
 \sum_{j>\bar d_t}\rise{(j+\delta)}{q}|v_j(t)|
 \le C_qt^{q+1-p}\le C_qt^{q\alpha},
\]
and $v(t)-t\pi=-c_0\pi-d_0A\pi$ has bounded weighted norm.
This proves the asserted bound for $e(t)$.  Since $p>1$, we may take
$q=1$; comparability of $j$ and $j+\delta$ for $j\ge m$ gives the
last assertion.
\end{proof}

\begin{lemma}\label{lem:mean-tail}
For every $x_0>0$ there is a constant $C_{x_0}$ such that, for all
sufficiently large $n$ and all $x\ge x_0$,
\[
 |\E\widehat\nu_n(x,\infty)-x^{-p}|
 \le C_{x_0}(a_n^{-1}+k^{-1}).
\]
For all sufficiently large $n$ and all $x\ge1$,
\[
 \E\widehat\nu_n(x,\infty)
 \le C\{x^{-p}+k^{\alpha-1}x^{-1}\}.
\]
Furthermore,
\begin{equation}\label{eq:mean-hill-centering}
 \sqrt k\,|\E\widehat\nu_nh|
 \le C\left\{\frac{\sqrt k}{a_n}
             +\frac{1+\log k}{\sqrt k}\right\}.
\end{equation}
In particular, $\sqrt k\,\E\widehat\nu_nh\to0$ if
$\sqrt k/a_n\to0$.
\end{lemma}

\begin{proof}
Let $\bar\nu_n(x)=(n/k)\sum_{j>a_nx}\pi_j$.  Applying the
gamma-ratio expansion \citep[\S5.11(iii)]{dlmf} to \eqref{eq:pi-tail} with
$d=\lfloor a_nx\rfloor+1$, and using
$(n/k)\beta^pa_n^{-p}=1$, gives
\begin{equation}\label{eq:mean-stationary-tail-bound}
 |\bar\nu_n(x)-x^{-p}|
 \le Ca_n^{-1}x^{-p-1},\qquad a_nx\ge m.
\end{equation}
Since
\[
 \E\widehat\nu_n(x,\infty)
 =\bar\nu_n(x)+\frac1k\sum_{j>a_nx}e_j(n),
\]
the unweighted bound in Lemma~\ref{lem:mean-profile} proves the first
assertion.  The first-moment bound in that lemma gives
\[
 \frac1k\sum_{j>a_nx}|e_j(n)|
 \le\frac{Cn^\alpha}{ka_nx}
 =Ck^{\alpha-1}x^{-1},
\]
which proves the second assertion.

By the integral representation of the logarithm,
\[
 \E\widehat\nu_nh
 =\int_1^\infty\bar\nu_n(x)\frac{\dd x}{x}
   -\alpha\bar\nu_n(1)
   +\frac1k\sum_{j>a_n}e_j(n)h(j/a_n).
\]
The first two terms have sum $O(a_n^{-1})$, by
\eqref{eq:mean-stationary-tail-bound} and
$\int_1^\infty x^{-p}\dd x/x=\alpha$.
For the last term, put
$B_n=\sum_j|e_j(n)|\le C$ and
$M_n=\sum_jj|e_j(n)|\le Cn^\alpha$.
When $B_n>0$, Jensen's inequality gives
\[
 \sum_j|e_j(n)|\log(1+j/a_n)
 \le B_n\log\left(1+\frac{M_n}{B_na_n}\right)
 \le C\log(1+Cn^\alpha/a_n)
 \le C(1+\log k).
\]
The middle inequality follows from the monotonicity of
$b\mapsto b\log(1+M_n/(ba_n))$.  When $B_n=0$, the same bound is
immediate.  Since
$|h(j/a_n)|\le\log(1+j/a_n)+\alpha$, the last term is
$O((1+\log k)/k)$.  This proves
\eqref{eq:mean-hill-centering} throughout $\alpha\in(0,1)$.
\end{proof}

\subsection{One-edge martingales and a variance bound}
Let $e=(t,i)$ denote an edge in chronological order, and write
$D_v(e-)=D_v(t,i-1)$, $\F_{e-}=\F_{t,i-1}$ and
$\F_e=\F_{t,i}$. For an eligible vertex, the one-degree transition at
$e$ increments its degree $d$ with probability $(d+\delta)/W_e$.
Let $P_{s,t}$ denote this one-degree transition operator between
edge or integer times. In particular, $P_{e,n}$ is its operator from immediately after $e$
to the end of step $n$. Its kernels are deterministic. For bounded
$f$, put $f_n(d)=f(d/a_n)$ and define
\[
 \xi^f_{n,e}(d)=P_{e,n}f_n(d+1)-P_{e,n}f_n(d),\qquad
 \rho_{v,e}=\frac{D_v(e-)+\delta}{W_e},\qquad
 \xi^f_{v,e}=\xi^f_{n,e}(D_v(e-)).
\]
The difference defining $\xi^f_{n,e}(d)$ is used for
$m\le d\le m(t-2)+i-1$, the admissible degrees immediately before
$e=(t,i)$. Both continuation states $d$ and $d+1$ are then admissible
immediately after $e$. When evaluating this function on a comparison
process, define $\xi^f_{n,e}(d)=0$ outside this range. This extension
does not change any quantity evaluated on the PA graph.
All sums over $v$ at $e=(t,i)$ are over $v<t$.
Write $\bar\xi^f_e=\sum_{v<t}\rho_{v,e}\xi^f_{v,e}$ and set
\begin{equation}\label{eq:variance-sums}
\begin{aligned}
 I_n(f,g)&=\frac1k\sum_e\sum_{v<t}\rho_{v,e}\xi^f_{v,e}\xi^g_{v,e},
 &K_n(f,g)&=\frac1k\sum_e\bar\xi^f_e\bar\xi^g_e,\\
 \Xi_v^f&=\sum_{e:\,t>b_v}\rho_{v,e}(1-\rho_{v,e})(\xi^f_{v,e})^2,
 &Q_n(f)&=\frac1k\sum_e\sum_{v<t}\rho_{v,e}^2(\xi^f_{v,e})^2.
\end{aligned}
\end{equation}
Here $b_v=\max(v,2)$, so the two seed vertices also start at a
deterministic degree.

\begin{lemma}\label{lem:martingale}
The Doob martingale $\E[S_n(f)\mid\F_e]$ has increments
\begin{equation}\label{eq:martingale-increment}
 \Delta^f_e=\frac1{\sqrt{k}}\sum_{v<t}
       \{\one\{J_e=v\}-\rho_{v,e}\}\xi^f_{v,e},
 \qquad |\Delta^f_e|\le4\|f\|_\infty/\sqrt{k}.
\end{equation}
Its predictable covariance with the martingale for $g$ is
$I_n(f,g)-K_n(f,g)$. Moreover,
\begin{equation}\label{eq:vertex-variance}
 \E\Xi_v^f=\Var f_n(D_v(n)),\qquad
 I_n(f,f)=\frac1k\sum_v\Xi_v^f+Q_n(f).
\end{equation}
With $c_*=r/(2mp)>0$, for every bounded measurable $f$,
\begin{equation}\label{eq:variance-domination}
 I_n(f,f)\le\frac1{c_*k}\sum_v\Xi_v^f,\qquad
 \E S_n(f)^2\le\frac1{c_*k}\sum_v\Var f_n(D_v(n))
                    \le c_*^{-1}\E\widehat\nu_nf^2.
\end{equation}
\end{lemma}

\begin{proof}
For a vertex already born, its terminal conditional mean is
$P_{e,n}f_n(D_v(e))$. For a vertex whose birth step is not completed,
its degree when it becomes eligible is deterministically $m$, so its
terminal conditional mean is constant until then. Such vertices make
no contribution to an increment. The one-step backward recursion
therefore gives~\eqref{eq:martingale-increment}.
Conditionally on $\F_{e-}$, the endpoint indicators are one categorical
trial, with covariance
$\rho_{v,e}\one\{u=v\}-\rho_{u,e}\rho_{v,e}$.
This gives the asserted predictable covariance. Applied to the
one-vertex terminal variable, the same calculation and orthogonality
of martingale differences give the first identity
in~\eqref{eq:vertex-variance}. The second follows by writing
$\rho=\rho(1-\rho)+\rho^2$.

Every other eligible vertex has weight at least $r$. Thus, before
an edge in step $t\ge3$,
\[
 1-\rho_{v,e}\ge\frac{r(t-2)}{W_e}
       \ge\frac{r(t-2)}{mp(t-1)}\ge c_*.
\]
It follows termwise that $\rho\le c_*^{-1}\rho(1-\rho)$, which proves
the first inequality in~\eqref{eq:variance-domination}. The expected
predictable variance is $\E S_n(f)^2$ and $K_n(f,f)\ge0$.
Taking expectations and using~\eqref{eq:vertex-variance} proves the
remaining bounds.
\end{proof}

The difference $\xi^f_{n,e}(d)$ is the change in the conditional
mean of the terminal value of $f_n$ caused by one additional degree
after edge $e$. A monotone coupling of two one-degree chains controls
it. Set $\lambda_{n,e}=R_1(e,n)/a_n$, where the product is over edges
strictly after $e$. Lemma~\ref{lem:clock} gives
$\lambda_{n,e}\le C(k/t)^\alpha$ for $e=(t,i)$.

\begin{lemma}\label{lem:one-degree-difference}
For a degree $d$ admissible immediately before $e$ and bounded
Lipschitz $f$, $|\xi^f_{n,e}(d)|\le
\Lip(f)\lambda_{n,e}$, and $\xi^f_{n,e}(d)\ge0$ if $f$ is
nondecreasing. If also $f=0$ on $(0,x_0)$ for some $x_0>0$, then for every integer
$\ell\ge0$,
\begin{equation}\label{eq:xi-bound}
 |\xi^f_{n,e}(d)|\le C_\ell\Lip(f)\lambda_{n,e}
      \min\left\{1,\left(\frac{(d+1)\lambda_{n,e}}{x_0}\right)^\ell\right\}.
\end{equation}
For fixed bounded Lipschitz $f$, $Q_n(f)\to0$ in $L^1$.
\end{lemma}

\begin{proof}
Start one-degree chains $Y,Y'$ just after $e$ at $d,d+1$ and use the
same uniforms at all future edges. Their order is preserved. Their
difference increases by one with conditional probability
$(Y'-Y)/W$, so
$\E(Y'_n-Y_n)=R_1(e,n)$ and
$\E(Y'_n-Y_n)^2\le C R_1(e,n)^2$ by the factorial-moment identity.
The Lipschitz and monotonicity assertions follow. If $f$ vanishes
below $x_0$, the terminal difference of the two test functions is zero
unless $Y'_n\ge x_0a_n$. The moment estimate of
Lemma~\ref{lem:moments}, followed by Cauchy--Schwarz, bounds its mean
by
\[
 \frac{\Lip(f)}{a_n}
 \{\E(Y'_n-Y_n)^2\}^{1/2}
 \PP(Y'_n\ge x_0a_n)^{1/2},
\]
and the moment of order $2\ell$ gives~\eqref{eq:xi-bound}.

Fix $\varepsilon>0$. For labels $v<\varepsilon k$,
$\rho^2\le c_*^{-1}\rho(1-\rho)$ and
$\E\Xi_v^f\le\|f\|_\infty^2$ give a contribution to
$\E Q_n(f)$ bounded by $C(\varepsilon+k^{-1})$.
For $v\ge\varepsilon k$, the Lipschitz estimate and the second degree
moment give, at $e=(t,i)$,
\[
 \E[\rho_{v,e}^2(\xi^f_{v,e})^2]
 \le Ct^{-2}(t/v)^{2\alpha}(k/t)^{2\alpha}
 =Ck^{2\alpha}v^{-2\alpha}t^{-2}.
\]
Summing over $t>v$, over the $m$ edges in each step, and then over
$v\ge\varepsilon k$ yields
\[
 \E Q_n(f)\le C\{\varepsilon+k^{-1}+\varepsilon^{-2\alpha}/k\}.
\]
Let $n\to\infty$ and then $\varepsilon\downarrow0$.
\end{proof}

\subsection{The limiting covariance}
The normalized degree of a vertex born at time close to $sk$ has
limiting law
$X_s=\beta^{-1}s^{-\alpha}G_r$. Its law determines the diagonal
part of the covariance. For a test function $f$, put
\begin{equation}\label{eq:omega}
 \omega_f(s)=r\,\E\{f(\beta^{-1}s^{-\alpha}G_{r+1})-f(X_s)\},
 \qquad \Omega_f(u)=\int_0^u\omega_f(s)\dd s.
\end{equation}
Define the covariance functional
\begin{equation}\label{eq:covariance-functional}
 V(f,g)=\int_0^\infty\Cov(f(X_s),g(X_s))\dd s
       -\frac1{mp^2}\int_0^\infty
                    \frac{\Omega_f(u)\Omega_g(u)}{u^2}\dd u.
\end{equation}

We record its integrability before taking limits. For measurable
$\varphi\ge0$ and $a>0$, changing variables at a fixed value of $G_a$
gives
\begin{equation}\label{eq:gamma-integral}
 \int_0^\infty\E\varphi(\beta^{-1}s^{-\alpha}G_a)\dd s
       =\beta^{-p}\E G_a^p\,\nu\varphi.
\end{equation}
In particular $\int_0^\infty\E f(X_s)^2\dd s=\nu f^2$.
For fixed $C_0,x_0>0$, if $|f|$ is bounded by
$F_0(x)=C_0(1+\log_+x)\one\{x\ge x_0\}$, then
\begin{equation}\label{eq:omega-bound}
 |\Omega_f(u)|\le C\min\{u(1+\log_+(1/u)),1\}.
\end{equation}
For $u\le1$, this follows from
$\log_+(\beta^{-1}s^{-\alpha}y)
\le\alpha\log(1/s)+\log_+(y/\beta)$. For $u\ge1$ it follows
from~\eqref{eq:gamma-integral} applied to both Gamma variables.
The absolute value of the first integrand
in~\eqref{eq:covariance-functional} is bounded by
$2\E F_0(X_s)^2$, whose integral is finite
by~\eqref{eq:gamma-integral}. The second integrand is bounded by
$C\min\{(1+\log_+(1/u))^2,u^{-2}\}$, which is integrable on
$(0,\infty)$. These bounds prove finiteness for $|f|,|g|\le F_0$.
They also show by dominated convergence that
$V(f_j,f_j)\to V(f,f)$ if $|f_j|\le F_0$ and $f_j\to f$
Lebesgue almost everywhere. The Gamma laws have densities, so null
sets cause no exception inside the expectations.

\begin{lemma}\label{lem:vertex-variance-average}
Let $x_0>0$ and let $f$ be bounded and Lipschitz and vanish on $(0,x_0)$.
Then, in probability in $\R$,
\[
 \frac1k\sum_v\Xi_v^f\longrightarrow
             \int_0^\infty\Var f(X_s)\dd s.
\]
\end{lemma}

\begin{proof}
Fix $0<\varepsilon<L<\infty$. Along any sequence of labels with
$v/k\to s\in[\varepsilon,L]$, Lemma~\ref{lem:degree-limit} gives
$D_v(n)/a_n\convd X_s$, and $s\mapsto\Var f(X_s)$ is continuous.
Since every sequence in $[\varepsilon,L]$ has a convergent
subsequence, it follows that
\begin{equation}\label{eq:uniform-marginal-variance}
 \sup_{\varepsilon k\le v\le Lk}
 |\Var f_n(D_v(n))-\Var f(X_{v/k})|\longrightarrow0.
\end{equation}
For an integer $q$ with $q\alpha>1$, the degree moments give
\[
 \frac1k\sum_{v>Lk}\Var f_n(D_v(n))
 \le\frac Ck\sum_{v>Lk}(k/v)^{q\alpha}
 \le CL^{1-q\alpha}.
\]
The labels $v<\varepsilon k$ contribute at most
$C(\varepsilon+k^{-1})$ in expectation.

For labels $u\ne v$ in $[\varepsilon k,Lk]$ we show that
$\Cov(\Xi_u^f,\Xi_v^f)\to0$ uniformly. At $e=(t,i)$,
\[
 0\le\rho_{v,e}(1-\rho_{v,e})(\xi^f_{v,e})^2
        \le C(D_v(e-)+\delta)t^{-1}(k/t)^{2\alpha}.
\]
Minkowski's inequality and fourth degree moments therefore give
$\sup_{\varepsilon k\le v\le Lk}\|\Xi_v^f\|_4\le C_\varepsilon$.
For fixed $T>L$, let $\Xi_v^{f,T}$ be the sum defining $\Xi_v^f$
restricted to steps $t\le Tk$. Then
\begin{equation}\label{eq:vertex-variance-time-tail}
 \sup_{\varepsilon k\le v\le Lk}
 \|\Xi_v^f-\Xi_v^{f,T}\|_4
 \le C_\varepsilon k^\alpha\sum_{t>Tk}t^{-\alpha-1}
 \le C_\varepsilon T^{-\alpha}.
\end{equation}
Up to time $Tk$, the coupling bound~\eqref{eq:finite-coupling}
couples two degree paths with independent comparison paths, with
failure probability at most $C_{\varepsilon,L,T}/k$.
Let $\widetilde\Xi_v^{f,T}$ denote the sum obtained by replacing
$D_v(e-)$ with $\widetilde D_v(e-)$, using the zero extension of
$\xi^f_{n,e}$. Each summand is nonnegative and bounded by
$C\{\widetilde D_v(e-)+\delta\}t^{-1}(k/t)^{2\alpha}$.
Lemma~\ref{lem:moments} and Minkowski's inequality therefore give
the same uniform fourth-moment bound for these sums.
The variables $\widetilde\Xi_u^{f,T}$ and
$\widetilde\Xi_v^{f,T}$ are independent because they are deterministic
functions of independent comparison paths.

The products of the two original variables and of the two comparison
variables agree whenever the coupling succeeds. H\"older's inequality
bounds the difference of their expectations by
$C_{\varepsilon,L,T}k^{-1/2}$, using the bounded second moments of the products
and the failure probability $C_{\varepsilon,L,T}/k$.
The individual means differ by at most
$C_{\varepsilon,L,T}k^{-3/4}$, using the fourth moments.
Subtracting the products of means consequently gives
\begin{equation}\label{eq:coupling-covariance}
 \sup_{\substack{u\ne v\\u,v\in[\varepsilon k,Lk]}}
 |\Cov(\Xi_u^{f,T},\Xi_v^{f,T})|
 \le C_{\varepsilon,L,T}k^{-1/2}.
\end{equation}
Restoring the tails
in~\eqref{eq:vertex-variance-time-tail} yields
\[
 \limsup_n\sup_{\substack{u\ne v\\u,v\in[\varepsilon k,Lk]}}
       |\Cov(\Xi_u^f,\Xi_v^f)|\le C_{\varepsilon,L}T^{-\alpha}.
\]
Letting $T\to\infty$ makes this bound zero.
The diagonal terms in the variance of
$k^{-1}\sum_{\varepsilon k\le v\le Lk}\Xi_v^f$ are bounded by
$C_{\varepsilon,L}/k$, and the off-diagonal terms tend to zero by
the uniform covariance bound. Its mean converges by
\eqref{eq:vertex-variance} and~\eqref{eq:uniform-marginal-variance}
to $\int_\varepsilon^L\Var f(X_s)\dd s$. Finally let
$\varepsilon\downarrow0$ and $L\to\infty$, using the nonnegative
tail bounds and~\eqref{eq:gamma-integral}.
\end{proof}

The second term of~\eqref{eq:covariance-functional} is the limit of
$K_n(f,f)$. To identify it, set
$\zeta^f_{v,e}=(D_v(e-)+\delta)\xi^f_{v,e}$, so that
$\bar\xi^f_e=W_e^{-1}\sum_{v<t}\zeta^f_{v,e}$.

\begin{lemma}\label{lem:weighted-average}
For $x_0>0$, bounded Lipschitz $f$ vanishing on $(0,x_0)$ and
$0<u_0<u_1<\infty$,
\[
 \max_{1\le i\le m}\sup_{u_0\le u\le u_1}
 \E\left[\bar\xi^f_{(\lfloor uk\rfloor,i)}
                  -\frac{\Omega_f(u)}{mpu}\right]^2\longrightarrow0.
\]
For all edges $e=(t,i)$,
$\E(\bar\xi^f_e)^2\le C\min\{1,(k/t)^2\}$.
Consequently, in $L^1$,
\[
 K_n(f,f)\longrightarrow
           \frac1{mp^2}\int_0^\infty\Omega_f(u)^2u^{-2}\dd u.
\]
\end{lemma}

\begin{proof}
We first identify $\E\zeta^f_{v,e}$. On functions of
one degree, write
$\mathcal B\varphi(d)=(d+\delta)\{\varphi(d+1)-\varphi(d)\}$.
Every one-edge transition is $I+\mathcal B/W_e$, so these operators
commute algebraically. For a vertex born at $v\ge3$, let
$P_{v,n}^{\setminus e}$ be the product of its transitions with just
edge $e$ omitted. Moving $\mathcal B$ through the transitions before
$e$ gives the exact identity
\begin{equation}\label{eq:weighted-identity}
 \E\zeta^f_{v,e}
   =r\{P_{v,n}^{\setminus e}f_n(m+1)
                        -P_{v,n}^{\setminus e}f_n(m)\}.
\end{equation}
This calculation uses the unextended difference only at reachable
states. Before an edge $(t,i)$, a chain started at $m+1$ at time $v$
has degree at most $m+1+m(t-v-1)+i-1$, which does not exceed
$m(t-2)+i-1$ because $m(v-2)\ge1$. Thus the chains started at $m$
or $m+1$, with or without the omitted edge, have valid transition
probabilities at every step. Couple a full chain with the chain skipping $e$
by common uniforms at all other edges. They agree unless the full
chain moves at $e$, whose probability is at most
$C(t/v)^\alpha/t$. For $\varepsilon k\le v<t\le u_1k$ this is
$O(k^{-1})$. Therefore~\eqref{eq:weighted-identity} and the fixed-start
Gamma limits in Lemma~\ref{lem:degree-limit} imply
\begin{equation}\label{eq:weighted-mean}
 \sup_{\varepsilon k\le v<t\le u_1k,\;1\le i\le m}
       |\E\zeta^f_{v,(t,i)}-\omega_f(v/k)|\longrightarrow0.
\end{equation}
Uniformity follows by taking convergent subsequences of $v/k$.

Lemma~\ref{lem:one-degree-difference} and~\eqref{eq:degree-moments} give
\[
 \|\zeta^f_{v,e}\|_4\le C(k/b_v)^\alpha,\qquad
 \|\zeta^f_{v,e}\|_2\le C_\ell(k/b_v)^{(\ell+1)\alpha}
 \quad(\ell\ge1).
\]
Choose $\ell$ with $(\ell+1)\alpha>1$ and $0<\varepsilon<u_0/2$.
For $t\ge u_0k$,
Minkowski's inequality shows
\[
 \left\|W_e^{-1}\sum_{v<\varepsilon k}\zeta^f_{v,e}\right\|_2
 \le\frac C{u_0k}\sum_{v<\varepsilon k}(k/b_v)^\alpha
 \le C_{u_0}\{\varepsilon^{1-\alpha}+k^{\alpha-1}\}.
\]
For distinct $v,w\in[\varepsilon k,t)$ with $t\le u_1k$, define
$\widetilde\zeta^f_{v,e}=(\widetilde D_v(e-)+\delta)
\xi^f_{n,e}(\widetilde D_v(e-))$ and define
$\widetilde\zeta^f_{w,e}$ in the same way. These two variables are
independent. Their fourth moments are uniformly bounded by
Lemma~\ref{lem:moments} and the zero extension of $\xi^f_{n,e}$.
They agree with their PA counterparts outside a coupling failure
of probability at most $C_{\varepsilon,u_1}/k$.
Bounding the difference of product expectations and the differences
of means by H\"older's inequality, as in the derivation
of~\eqref{eq:coupling-covariance}, gives
$|\Cov(\zeta^f_{v,e},\zeta^f_{w,e})|\le C_{\varepsilon,u_1}k^{-1/2}$.
There are at most $u_1k$ summands and $W_e\ge c u_0k$.
The variance of their sum divided by $W_e^2$ is therefore at most
$C_{\varepsilon,u_0,u_1}(k^{-1}+k^{-1/2})$.
By~\eqref{eq:weighted-mean},
its mean tends uniformly to
$(mpu)^{-1}\int_\varepsilon^u\omega_f(s)\dd s$.
Let $\varepsilon\downarrow0$, using~\eqref{eq:omega-bound}.
This proves the first assertion.

For $t\le k$, $|\bar\xi_e^f|\le2\|f\|_\infty$.
For $t>k$, split the sum at $v=k$ and use the two moment bounds:
\[
 \|\bar\xi_e^f\|_2
 \le\frac Ct\left\{\sum_{v\le k}(k/b_v)^\alpha
       +\sum_{v>k}(k/v)^{(\ell+1)\alpha}\right\}
 \le Ck/t.
\]
This proves the second assertion. On $[u_0,u_1]$ the function
$\Omega_f(u)/(mpu)$ is continuous and bounded. The first assertion
and Cauchy--Schwarz show, uniformly over this interval and the edge
index $i$, that the squares of $\bar\xi^f_{(\lfloor uk\rfloor,i)}$
converge in $L^1$ to $\{\Omega_f(u)/(mpu)\}^2$.
Riemann sums then give convergence of $K_n$ over
$u_0k\le t\le u_1k$.
The two omitted time ranges have expected contributions at most
$Cu_0$ and $C/u_1$. Let $u_0\downarrow0$ and $u_1\to\infty$.
There are $m$ edges per step and each limit of $\bar\xi_e^f$ has
denominator $mpu$, which gives the coefficient $1/(mp^2)$.
\end{proof}

\begin{proposition}\label{prop:smooth-clt}
For $x_0>0$ and bounded Lipschitz $f$ vanishing on $(0,x_0)$,
$S_n(f)\convd N(0,V(f,f))$ in $\R$.
\end{proposition}

\begin{proof}
By Lemmas~\ref{lem:one-degree-difference}, \ref{lem:vertex-variance-average}
and~\ref{lem:weighted-average}, the predictable variance
$I_n(f,f)-K_n(f,f)$ converges in probability to $V(f,f)$.
This limit is nonnegative, since each predictable variance is
nonnegative. The mean convergence proved in
Lemma~\ref{lem:vertex-variance-average}, together
with~\eqref{eq:vertex-variance} and $\E Q_n(f)\to0$, also gives
$\E I_n(f,f)\to\int_0^\infty\Var f(X_s)\dd s$.
Since $K_n(f,f)$ converges in $L^1$,
$\E S_n(f)^2\to V(f,f)$. If $V(f,f)=0$, the conclusion follows
in $L^2$. Suppose now that $V(f,f)>0$.
For every $\eta>0$, equation~\eqref{eq:martingale-increment}
gives, for all sufficiently large $n$,
\[
 \sum_e\E\left[(\Delta_e^f)^2\one\{|\Delta_e^f|>\eta\}
                       \mid\F_{e-}\right]=0.
\]
Thus the conditional Lindeberg condition holds. The graph can be
constructed as one infinite process, so the filtration at a fixed
edge is identical in all rows of the martingale array containing
that edge. The martingale central limit theorem
\citep[\S3.2]{hall1980martingale} applies to this triangular array.
The initial sigma-field is trivial and the terminal value is $S_n(f)$,
so the claimed convergence follows.
\end{proof}

\subsection{Discontinuous functions and the Hill statistic}
The bound~\eqref{eq:variance-domination} allows bounded
Lipschitz test functions to approximate functions with jumps. If $f$
is bounded, has compact support in $(0,\infty)$ and is Lipschitz
between finitely many jump points, let $f_\eta$ replace $f$ by linear
interpolation on intervals of length $2\eta$ around those points.
For sufficiently small $\eta$, all these functions vanish on a
common interval $(0,x_0)$ with $x_0>0$ and have a common bounded
support. The difference is bounded by $2\|f\|_\infty$ and supported
on the interpolation intervals. Lemmas~\ref{lem:mean-tail}
and~\ref{lem:martingale}, applied to slightly larger intervals to
include any lattice atoms at their endpoints, give
\[
 \limsup_n\E\{S_n(f)-S_n(f_\eta)\}^2\le C_f\eta.
\]
Proposition~\ref{prop:smooth-clt} applies to $f_\eta$, and
\eqref{eq:omega-bound} gives $V(f_\eta,f_\eta)\to V(f,f)$.
The approximation theorem for weak convergence in
\citet{billingsley1999convergence} yields
$S_n(f)\convd N(0,V(f,f))$ in $\R$.

For $A>\alpha$, put $h_A=h\one_{(0,e^A]}$.
This convergence applies to $h_A$.
For each $n$, all degrees are at most $m(n-1)$, so
$h-h_A$ may be replaced by a bounded function agreeing with it on
$(0,m(n-1)/a_n]$. The constant in~\eqref{eq:variance-domination}
is independent of that function, so the inequality applies for each
$n$. Together with the tail bound of Lemma~\ref{lem:mean-tail}, it gives
\begin{equation}\label{eq:hill-truncation}
\begin{split}
 \E\{S_n(h)-S_n(h_A)\}^2
 &\le C\,\E\widehat\nu_n
           [(\log x-\alpha)^2\one\{x>e^A\}]\\
 &\le C(1+A^2)e^{-pA}
           +Ck^{\alpha-1}(1+A^2)e^{-A}.
\end{split}
\end{equation}
The second inequality follows by expressing the squared logarithm as
its value at $e^A$ plus the integral of
$2(\log x-\alpha)/x$ from $e^A$ onwards. Since $\alpha<1$, its
right side tends to zero as $n\to\infty$ and then $A\to\infty$.
The bound~\eqref{eq:omega-bound} and dominated convergence also give
$V(h_A,h_A)\to V(h,h)$. The approximation theorem in
\citet{billingsley1999convergence}, applied again, proves
\begin{equation}\label{eq:hill-centered-clt}
 S_n(h)\convd N(0,V(h,h))\quad\hbox{in }\R.
\end{equation}
No condition beyond $k\to\infty$ and $k/n\to0$ has been used for
this centered empirical limit.

The random threshold requires the stronger condition
$\sqrt{k}/a_n\to0$, equivalent to~\eqref{eq:k-rate}.

\begin{lemma}\label{lem:local-threshold}
Under~\eqref{eq:k-rate}, for each fixed $0<C_0<\infty$,
\begin{equation}\label{eq:local-linearization}
 \sup_{|u|\le C_0}\left|
 \sqrt{k}\{\widehat\nu_n(1+u/\sqrt{k},\infty)
                   -\widehat\nu_n(1,\infty)\}+pu\right|\convp0.
\end{equation}
Moreover,
\begin{equation}\label{eq:threshold-linearization}
 \sqrt{k}(\widehat x_n-1)
       -\alpha\sqrt{k}\{\widehat\nu_n(1,\infty)-1\}\convp0,
\end{equation}
and the last line of~\eqref{eq:hill-decomposition} tends to zero in
probability in $\R$.
\end{lemma}

\begin{proof}
Write $f_x=\one_{(x,\infty)}$.
For $1/2\le x<y\le2$, \eqref{eq:variance-domination} and
Lemma~\ref{lem:mean-tail} give
\begin{equation}\label{eq:local-increments}
 \E\{S_n(f_x)-S_n(f_y)\}^2
       \le C\{y-x+a_n^{-1}+k^{-1}\}.
\end{equation}
Also $\sup_n\E S_n(f_1)^2<\infty$.
For a fixed integer $N$, choose a uniform grid
$u_j=-C_0+2jC_0/N$, $0\le j\le N$, and put $x_j=1+u_j/\sqrt{k}$.
Equation~\eqref{eq:local-increments} makes
$\max_j|S_n(f_{x_j})-S_n(f_1)|\convp0$.
For $x_j\le x\le x_{j+1}$, monotonicity of the empirical tail and its
mean bounds $|S_n(f_x)-S_n(f_1)|$ by this maximum plus
\[
 \sqrt{k}\{\E\widehat\nu_n(x_j,\infty)
                  -\E\widehat\nu_n(x_{j+1},\infty)\}
 \le \frac{2pC_0}{N}+o(1)
                  +C(\sqrt{k}/a_n+k^{-1/2}).
\]
Let $n\to\infty$ and then $N\to\infty$.
The centered increments vanish uniformly in $|u|\le C_0$.
The corresponding mean increment is $-pu+o(1)$ uniformly, because
Lemma~\ref{lem:mean-tail} gives a uniform error
$O(\sqrt{k}/a_n+k^{-1/2})$, and a Taylor expansion of
$(1+u/\sqrt{k})^{-p}$ gives the stated linear term.
This proves~\eqref{eq:local-linearization}.

Put $U_n=\sqrt{k}\{\widehat\nu_n(1,\infty)-1\}$.
Including ties, the order statistic satisfies
\begin{equation}\label{eq:threshold-event}
 \{\widehat x_n>x\}=\{\widehat\nu_n(x,\infty)>1\}.
\end{equation}
Fix $\epsilon>0$. Lemma~\ref{lem:mean-tail}, $\sup_n\E S_n(f_1)^2<\infty$
and Chebyshev's inequality give $C_1$ with
$\limsup_n\PP(|U_n|>C_1)\le\epsilon$.
On $|U_n|\le C_1$, for fixed $\eta>0$, set
$u_\pm=(U_n\pm2\eta)/p$. Equation~\eqref{eq:local-linearization}
makes the empirical tail at $1+u_+/\sqrt{k}$ smaller than one and
at $1+u_-/\sqrt{k}$ larger than one, with probability tending to at
least $1-\epsilon$. Equation~\eqref{eq:threshold-event} then gives
\[
 (U_n-2\eta)/p
    <\sqrt{k}(\widehat x_n-1)\le(U_n+2\eta)/p.
\]
Letting $\epsilon,\eta\downarrow0$
proves~\eqref{eq:threshold-linearization}.

To control the integral in~\eqref{eq:hill-decomposition}, fix
$\epsilon>0$. The preceding bounds give a constant $C_2$ such that,
with limiting probability at least $1-\epsilon$,
\[
 |U_n|\le C_2,\qquad
 \sqrt{k}|\widehat x_n-1|\le C_2,\qquad
 \sup_{|u|\le C_2}\left|
 \sqrt{k}\{\widehat\nu_n(1+u/\sqrt{k},\infty)
                    -\widehat\nu_n(1,\infty)\}+pu\right|\le1.
\]
On this event, $x\ge1/2$ between $1$ and $\widehat x_n$ for large
$n$. The empirical tail changes there by at most
$(pC_2+1)/\sqrt{k}$. Integrating this bound, and using
$|\log x-(x-1)|\le2(x-1)^2$ for $|x-1|\le1/2$, gives
\begin{equation}\label{eq:threshold-integral}
 \left|\sqrt{k}\int_1^{\widehat x_n}
            \widehat\nu_n(x,\infty)\frac{\dd x}{x}
                   -\sqrt{k}(\widehat x_n-1)\right|
 \le\frac{C(C_2)}{\sqrt{k}}.
\end{equation}
Let $n\to\infty$ and then $\epsilon\downarrow0$.
Equations~\eqref{eq:threshold-linearization}
and~\eqref{eq:threshold-integral} prove the final assertion.
\end{proof}

\subsection{Evaluation of the limiting variance}

Since $\int_0^\infty\E h(X_s)^2\dd s=\nu h^2$
by~\eqref{eq:gamma-integral}, the covariance
functional~\eqref{eq:covariance-functional} gives $V(h,h)$ as $\nu h^2$
minus two nonnegative terms. The first,
$\int_0^\infty\{\E h(X_s)\}^2\dd s$, arises because vertices born at
different times have different degree laws. The second, the integral
of $\Omega_h(u)^2u^{-2}$, arises from the competition among old
vertices for each new edge. The next lemma evaluates both terms.

\begin{lemma}\label{lem:variance-value}
Let $G,G'$ be independent Gamma variables of shape $r=m+\delta$.
With $\Lambda_{m,\delta}$ defined by~\eqref{eq:variance-function},
$\nu h^2=\alpha^2$ and
\begin{align}
 \int_0^\infty\{\E h(X_s)\}^2\dd s
 &=\Lambda_{m,\delta}
 =\alpha^2\frac{\E\{(G\wedge G')^p\}}{\E(G^p)},
 \label{eq:heterogeneity-value}\\
 \frac{1}{mp^2}\int_0^\infty\frac{\Omega_h(u)^2}{u^2}\dd u
 &=\frac{\Lambda_{m,\delta}}{m}.
 \label{eq:competition-value}
\end{align}
Consequently,
\[
 V(h,h)=\alpha^2-(1+m^{-1})\Lambda_{m,\delta}=V_m(\delta),
 \qquad 0<V_m(\delta)<\alpha^2,
\]
where $V_m(\delta)$ is given by \eqref{eq:variance-function} and is
continuous for $\delta>-m$.
The representation~\eqref{eq:variance-beta} also holds.
\end{lemma}

\begin{proof}
For $x\le1$, $\Gamma(r,x)\le\Gamma(r)$, and for $x\ge1$,
$\Gamma(r,x)\le2^r\Gamma(r)e^{-x/2}$. If $\delta$ ranges over a
fixed compact subset of $(-m,\infty)$, choose $1<p_-\le p\le p_+<\infty$.
The integrand in~\eqref{eq:variance-function}, including its
normalizing factor, is bounded by $Cx^{p_--1}$ on $(0,1]$ and by
$Cx^{p_+-1}e^{-x}$ on $[1,\infty)$. Both bounds are integrable.
Dominated convergence proves finiteness and continuity of
$\Lambda_{m,\delta}$, and hence continuity of $V_m$.
The change of variables $x=e^y$ gives
\[
 \nu h^2=\int_1^\infty(\log x-\alpha)^2px^{-p-1}\dd x
        =\int_0^\infty(y-\alpha)^2pe^{-py}\dd y=\alpha^2.
\]

To calculate the first integral, put $a=\beta s^\alpha$.  Since
$X_s>1$ exactly when $G>a$, independence gives
\[
 \{\E h(X_s)\}^2
 =\E\left[\{\log(G/a)-\alpha\}\{\log(G'/a)-\alpha\}
                      \one\{G\wedge G'>a\}\right].
\]
Hence the substitution $s=(a/\beta)^p$ and Fubini's theorem yield
\begin{equation}\label{eq:mean-square-integral}
 \int_0^\infty\{\E h(X_s)\}^2\dd s
 =p\beta^{-p}\E\int_0^{G\wedge G'}
       a^{p-1}\{\log(G/a)-\alpha\}\{\log(G'/a)-\alpha\}\dd a .
\end{equation}
Fubini's theorem applies because Cauchy--Schwarz and
\eqref{eq:gamma-integral} give
\[
 p\beta^{-p}\int_0^\infty a^{p-1}
       \E|h(G/a)h(G'/a)|\dd a
 \le p\beta^{-p}\int_0^\infty a^{p-1}\E h(G/a)^2\dd a
 =\nu h^2=\alpha^2.
\]
For $0<x\le y$, the substitution $a=xe^{-z}$ gives
\begin{align*}
 &\int_0^x a^{p-1}\{\log(x/a)-\alpha\}
                       \{\log(y/a)-\alpha\}\dd a\\
 &\qquad=x^p\int_0^\infty e^{-pz}(z-\alpha)
                             \{z-\alpha+\log(y/x)\}\dd z
       =\frac{x^p}{p^3}.
\end{align*}
Indeed, $\int_0^\infty e^{-pz}(z-\alpha)\dd z=0$ and
$\int_0^\infty e^{-pz}(z-\alpha)^2\dd z=p^{-3}$.
Substitution in \eqref{eq:mean-square-integral}, followed by
$\beta^p=\E G^p$, proves the ratio in
\eqref{eq:heterogeneity-value}.  Finally,
\[
 \E(G\wedge G')^p
 =\int_0^\infty pa^{p-1}\PP(G>a)^2\dd a
 =\frac{p}{\Gamma(r)^2}
       \int_0^\infty a^{p-1}\Gamma(r,a)^2\dd a,
\]
and $\beta^p=\Gamma(r+p)/\Gamma(r)$ identifies this ratio with
$\Lambda_{m,\delta}$.

For the competition term, use the density relation
$r\,\PP(G_{r+1}\in\dd y)=y\,\PP(G_r\in\dd y)$ in
\eqref{eq:omega}. With $a=\beta s^\alpha$, it gives
\[
 \omega_h(s)=\frac1{\Gamma(r)}\int_a^\infty
       (y-r)\{\log(y/a)-\alpha\}y^{r-1}e^{-y}\dd y.
\]
Since $(y^re^{-y})'=(r-y)y^{r-1}e^{-y}$, integration by parts yields
\begin{equation}\label{eq:omega-hill}
 \omega_h(s)=\frac{\Gamma(r,a)-\alpha a^re^{-a}}{\Gamma(r)}.
\end{equation}
The second term is the boundary contribution of the jump of $h$ at
$1$. It must be retained even though the Gamma laws have no atoms.
Substituting $s=(a/\beta)^p$ and using $p\alpha=1$ now gives
\begin{equation}\label{eq:integrated-omega-hill}
 \begin{split}
 \Omega_h(u)
 &=\frac{\beta^{-p}}{\Gamma(r)}
      \int_0^{\beta u^\alpha}
       \{pa^{p-1}\Gamma(r,a)-a^{r+p-1}e^{-a}\}\dd a\\
 &=\frac{(\beta u^\alpha)^p\Gamma(r,\beta u^\alpha)}{\Gamma(r+p)}
   =u\,\PP(X_u>1).
 \end{split}
\end{equation}
The integrand is the derivative of $a^p\Gamma(r,a)$, whose value
tends to zero as $a\downarrow0$. Write $a=\beta u^\alpha$ again.
Since $u^{-2}\dd u=p\beta^p a^{-p-1}\dd a$,
\[
 \frac1{mp^2}\int_0^\infty\frac{\Omega_h(u)^2}{u^2}\dd u
 =\frac{\beta^p}{mp\,\Gamma(r+p)^2}
        \int_0^\infty a^{p-1}\Gamma(r,a)^2\dd a
 =\frac{\Lambda_{m,\delta}}m.
\]
This proves \eqref{eq:competition-value}.  Subtracting the squared
mean from $\int\E h(X_s)^2\dd s=\nu h^2$ and then subtracting the
competition term gives the asserted value of $V(h,h)$.

The strict upper bound follows from $\Lambda_{m,\delta}>0$.  To prove
positivity, observe that, for $x,y>0$ and $p>1$,
$(x\wedge y)^p\le x^{p-1}y$, with equality only when $x=y$.
The independent Gamma variables have continuous laws, and therefore
\[
 \frac{\E(G\wedge G')^p}{\E G^p}
 <\frac{\E G^{p-1}\,\E G'}{\E G^p}
 =\frac{r\Gamma(r+p-1)}{\Gamma(r+p)}
 =\frac{r}{r+p-1}=\frac{m}{m+1}.
\]
The last identity uses $r=m(p-1)$.  This proves $V_m(\delta)>0$
throughout the admissible range $\delta>-m$.

For the beta representation, $T=G+G'$ and $U=G/(G+G')$ are
independent, with laws $\mathrm{Gamma}(2r,1)$ and
$\mathrm{Beta}(r,r)$.  Write
$B_x(a,b)=\int_0^x t^{a-1}(1-t)^{b-1}\dd t$ and
$B(a,b)=\Gamma(a)\Gamma(b)/\Gamma(a+b)$.  Symmetry about $1/2$ gives
\begin{align*}
 \frac{\E(G\wedge G')^p}{\E G^p}
 &=\frac{\Gamma(2r+p)\Gamma(r)}{\Gamma(2r)\Gamma(r+p)}
                   \frac{2B_{1/2}(r+p,r)}{B(r,r)}\\
 &=2\frac{B_{1/2}(r+p,r)}{B(r+p,r)}
 =2I_{1/2}(r+p,r).
\end{align*}
Substitution proves \eqref{eq:variance-beta}.
\end{proof}

\begin{proof}[Proof of Theorem~\ref{thm:clt}]
Equation~\eqref{eq:hill-centered-clt} gives the centered empirical
normal limit. Under~\eqref{eq:k-rate}, Lemma~\ref{lem:mean-tail}
shows $\sqrt{k}\E\widehat\nu_nh\to0$, and
Lemma~\ref{lem:local-threshold} makes the threshold remainder
in~\eqref{eq:hill-decomposition} vanish in probability.
Lemma~\ref{lem:variance-value} identifies the variance as
$V_m(\delta)$. These statements together prove the theorem.
\end{proof}

\begin{proof}[Proof of Corollary~\ref{cor:variance}]
Lemma~\ref{lem:variance-value} gives
$V_m(\delta)=\alpha^2-(1+m^{-1})\Lambda_{m,\delta}$ and
$0<V_m(\delta)<\alpha^2$.
\end{proof}

\section{Discussion}\label{sec:discussion}
The condition $\delta>-m$ is exactly the positivity condition for the
attachment weights. The restriction $\alpha<1$ makes the sum
in~\eqref{eq:summable-growth} finite, and in the proof of
Lemma~\ref{lem:weighted-average} it bounds the contribution of labels
below $\varepsilon k$ by $C\varepsilon^{1-\alpha}$. A pathwise
coupling over an unbounded time interval would need the sum of
$g(s,t)^2t^{-2}$ over $t>s$ to be small, and this sum diverges when
$\alpha\ge1/2$. Coupling paths only over bounded ratios of times, and
controlling $I_n(f,f)$ through the moments of $\Xi^f_v$, avoids a
stronger restriction on $\delta$.

Condition~\eqref{eq:k-rate} plays a different role. After
normalization the degree lattice has mesh $1/a_n$, whereas the
threshold $\widehat x_n$ fluctuates on the scale $k^{-1/2}$.
Condition~\eqref{eq:k-rate} makes the former negligible relative to
the latter, and it is necessary for the uniform local
approximation~\eqref{eq:local-linearization}.
Indeed, as a function of $u$, the empirical term
in~\eqref{eq:local-linearization} is constant between successive
points separated by $\sqrt{k}/a_n$. If this spacing is
bounded below on a subsequence, a fixed interval of positive length
contains a subinterval of length bounded below on which that term is
constant. Its uniform distance from the line $-pu$ is then bounded
away from zero. This argument concerns~\eqref{eq:local-linearization}
only. It does not show that~\eqref{eq:k-rate} is necessary for every
centered limit of the Hill estimator.

The results keep $m$ and $\delta$ fixed. The estimates are not
uniform as $\delta\downarrow-m$ and do not cover a growing number
of edges per vertex. They apply to the sequential attachment rule
in~\eqref{eq:attachment}. Extending the theorem to another attachment
convention requires checking its expected degree recursion and its
conditional covariance at each edge. The theorem also assumes a
deterministic sequence $k_n$; a choice made from the observed degree
sequence requires further analysis.

\section*{Funding}
T. Wang was funded by the National Natural Science Foundation
of China under Grant 12301660 and the Science and Technology Commission of Shanghai
Municipality under Grant 23JC1400700.

\appendix
\section{Degree growth and comparison processes}\label{sec:tools}
We compare degrees with independent linear pure-birth processes.
Their rates remain positive for $\delta>-m$. Comparing normalized
terminal values permits the full range $\alpha\in(0,1)$, while path
couplings are used only over bounded ratios of times. The products
$R_q$, $\Theta$ and $g$ are defined in~\eqref{eq:products}.

\subsection{Pure-birth processes and moments}

For an integer $d\ge m$, let
$B^{(\delta,d)}(\theta)$ start at $d$ and jump from $j$ to $j+1$ at
rate $j+\delta$. These rates are positive, and the process is
nonexplosive because $\sum_{j\ge d}(j+\delta)^{-1}=\infty$.
These linear pure-birth processes are standard branching-process
constructions \citep{athreya1972branching,athreya2008growth}. We record
their formulas for the possibly noninteger initial weight $d+\delta$.

\begin{lemma}\label{lem:birth}
For $j\ge0$, $q\ge0$ and $\theta\ge0$,
\begin{align}
 \PP\{B^{(\delta,d)}(\theta)=d+j\}
 &=\frac{\Gamma(d+\delta+j)}{\Gamma(d+\delta)j!}
 e^{-(d+\delta)\theta}(1-e^{-\theta})^j,
 \label{eq:birth-law}\\
 \E\rise{\{B^{(\delta,d)}(\theta)+\delta\}}{q}
 &=\rise{(d+\delta)}{q}e^{q\theta}.
 \label{eq:birth-moments}
\end{align}
The martingale $e^{-\theta}\{B^{(\delta,d)}(\theta)+\delta\}$ converges
in $L^2$ to a variable with law $G_{d+\delta}$, and
\begin{equation}\label{eq:birth-relative}
 \E\left[\left\{
 \frac{e^{-\theta}\{B^{(\delta,d)}(\theta)+\delta\}}{d+\delta}-1
 \right\}^{\!2}\right]
 =\frac{1-e^{-\theta}}{d+\delta}.
\end{equation}
\end{lemma}

\begin{proof}
The probabilities in \eqref{eq:birth-law} solve the forward equations
and sum to one by the negative-binomial series. Multiplying their
$j$th terms by $\Gamma(d+\delta+j+q)/\Gamma(d+\delta+j)$ and summing
gives \eqref{eq:birth-moments}, also for noninteger $q$.
The Markov property and the cases $q=1,2$ give the martingale
property, \eqref{eq:birth-relative}, and $L^2$ boundedness. Its
Laplace transform converges to $(1+\lambda)^{-(d+\delta)}$, since
\[
 \E e^{-\lambda e^{-\theta}(B^{(\delta,d)}(\theta)-d)}
 =\left\{\frac{e^{-\theta}}
 {1-(1-e^{-\theta})e^{-\lambda e^{-\theta}}}\right\}^{d+\delta}.
\]
This identifies the $L^2$ limit.
\end{proof}

\begin{lemma}\label{lem:clock}
For every fixed $q\ge0$ there is a constant $C_q$ such that, for all
integers $2\le s\le t$,
\begin{equation}\label{eq:clock-estimates}
 |\log R_q(s,t)-q\log g(s,t)|\le\frac{C_q}{s},\qquad
 R_q(s,t)\le C_q g(s,t)^q.
\end{equation}
In particular $C^{-1}g(s,t)\le R_1(s,t)\le Cg(s,t)$, and
\begin{equation}\label{eq:summable-growth}
 \sum_{u>s}\frac{g(s,u)}{u^2}\le\frac Cs.
\end{equation}
\end{lemma}

\begin{proof}
The denominators are positive and comparable to $u$ at every edge
of step $u$, with constants depending only on the fixed parameters.
At step $u$,
\[
 \sum_{i=1}^m W_{u,i}^{-1}
 =\frac{m}{(2m+\delta)u}+O(u^{-2}),\qquad
 \sum_{i=1}^m W_{u,i}^{-2}=O(u^{-2}).
\]
Expanding $\log(1+q/W_{u,i})$ and comparing the harmonic sum with
its integral proves \eqref{eq:clock-estimates}. Since $\alpha<1$,
summing $s^{-\alpha}u^{\alpha-2}$ proves
\eqref{eq:summable-growth}.
\end{proof}

Let $Y$ be a one-vertex degree chain, started from an admissible degree
$d$ at time $s$, whose increment at edge $e=(t,i)$ is Bernoulli with
probability $\{Y(e-)+\delta\}/W_e$. Let $\widetilde Y$ start
from the same degree and replace that increment by the increment of
$B^{(\delta,\widetilde Y(e-))}$ run for time $\log(1+W_e^{-1})$.
The corresponding transition operators are
$P_{s,t}$, already used in Section~\ref{sec:clt}, and $Q_{s,t}$. For each vertex $v$, define
$\widetilde D_v$ by these latter transitions, starting from $m$ at
$b_v=\max(v,2)$, and take the processes for different vertices
independent. These definitions also apply between individual edges.

\begin{lemma}\label{lem:moments}
For every real $q\ge0$,
\begin{equation}\label{eq:degree-moments}
 \E Y(t)^q+\E\widetilde Y(t)^q
 \le C_q(d+1)^qg(s,t)^q,
 \qquad
 \E D_v(t)^q+\E\widetilde D_v(t)^q
 \le C_q g(b_v,t)^q.
\end{equation}
The exact identities are
\begin{equation}\label{eq:exact-moments}
 \E\rise{(Y(t)+\delta)}{q}
 =\rise{(d+\delta)}{q}R_q(s,t),\qquad
 \E\rise{(\widetilde Y(t)+\delta)}{q}
 =\rise{(d+\delta)}{q}R_1(s,t)^q.
\end{equation}
\end{lemma}

\begin{proof}
For every $y>0$ and $q\ge0$,
$\rise{(y+1)}{q}-\rise{y}{q}=q\rise{y}{q}/y$.
Thus the conditional multiplier of $\rise{(Y+\delta)}{q}$ at one
edge is $1+q/W_e$. The comparison multiplier is
$(1+1/W_e)^q$ by \eqref{eq:birth-moments}. Iteration gives
\eqref{eq:exact-moments}. Ordinary powers of $d$ are uniformly comparable to
$\rise{(d+\delta)}{q}$ for $d\ge m$, because their ratio tends to
one as $d\to\infty$. Lemma~\ref{lem:clock} therefore proves
\eqref{eq:degree-moments}.
\end{proof}

\subsection{One-edge and terminal comparisons}

A coupling of one edge allocation with independent negative-binomial
increments controls the error in the comparison of terminal values.

\begin{lemma}\label{lem:one-edge}
Let $x>0$, $y_j>0$ for $1\le j\le h$, and
$y=\sum_jy_j\le1/x$. Let $K$ assign mass $y_jx$ to the $j$th
unit vector and mass $1-yx$ to zero. Let $Z_j$ be independent
negative-binomial increments with shape $y_j$ and success
probability $(1+x)^{-1}$. There is a coupling such that
\begin{equation}\label{eq:one-edge-comparison}
 \E|K_j-Z_j|\le y_j(y+1)x^2,\qquad
 \PP(K\ne Z)\le\E\|K-Z\|_1\le y(y+1)x^2.
\end{equation}
\end{lemma}

\begin{proof}
The sum $A=\sum_jZ_j$ is negative binomial with shape $y$, so
$\E A=yx$ and $\E A(A-1)=y(y+1)x^2$.
Conditional on $A=\ell$, the vector has probabilities
\[
 \frac{\ell!\,\Gamma(y)}{\Gamma(y+\ell)}
 \prod_{j=1}^h\frac{\Gamma(y_j+z_j)}{\Gamma(y_j)z_j!},
 \qquad \sum_jz_j=\ell,
\]
and hence $\E(Z_j\mid A=\ell)=\ell y_j/y$.
On $\{A\ge1\}$ choose $J$ with probabilities $Z_j/A$ and put
$K=e_J$. On a subset of $\{A=0\}$ of unconditional probability
$yx-\PP(A\ge1)$, choose $J$ with probabilities $y_j/y$ and also
put $K=e_J$. Otherwise put $K=0$. The extra probability is between
zero and $\PP(A=0)$ because $\PP(A\ge1)\le yx\le1$.
The resulting $K$ has the prescribed law, and
\[
 \E|K_j-Z_j|=\frac{2y_j}{y}\E(A-1)_+
 \le\frac{y_j}{y}\E A(A-1).
\]
Summing over $j$ gives the remaining bounds.
\end{proof}

\begin{lemma}\label{lem:comparison}
For integers $2\le s\le n$ and $m\le d\le m(s-1)$, the
following bound holds for every Lipschitz $\varphi$ on the integer
states, with a constant independent of $s,d,n$:
\begin{equation}\label{eq:kernel-comparison}
 |P_{s,n}\varphi(d)-Q_{s,n}\varphi(d)|
 \le C\Lip(\varphi)(d+1)^2\frac{R_1(s,n)}s.
\end{equation}
For fixed $h$, distinct labels $v_1,\ldots,v_h\le n$, and
$2\le u\le n$, the tracked PA degrees and the independent comparison
processes can be coupled up to time $u$ with mismatch probability
at most
\begin{equation}\label{eq:finite-coupling}
 C_h\sum_{j=1}^h\sum_{b_{v_j}<t\le u}
 \frac{g(b_{v_j},t)^2}{t^2}.
\end{equation}
For $v_j\in[\varepsilon k,Lk]$ and $u\le Tk$, with fixed
$0<\varepsilon<L<T$, this probability is at most $C_{h,\varepsilon,L,T}/k$.
\end{lemma}

\begin{proof}
Two pure-birth processes whose initial degrees differ by an integer
can be coupled monotonically so that their difference is a linear
birth process with rate equal to its current size. At one edge its
expected multiplier is $1+1/W_e$, independently of the offset.
It follows that
\begin{equation}\label{eq:comparison-lipschitz}
 \Lip(Q_{e,n}\varphi)\le R_1(e,n)\Lip(\varphi).
\end{equation}
Telescope the transition operators, using PA transitions before the
edge being replaced and comparison transitions afterward. All states
at the replaced edge are then admissible PA states, and no PA transition
is applied to an unbounded comparison state. Lemma~\ref{lem:one-edge}
and \eqref{eq:comparison-lipschitz} bound the error by
\[
 \Lip(\varphi)\sum_{e>s}R_1(e,n)W_e^{-2}
 \E[\{Y(e-)+\delta\}\{Y(e-)+\delta+1\}].
\]
By Lemmas~\ref{lem:clock} and~\ref{lem:moments}, this is at most
\[
 C\Lip(\varphi)(d+1)^2R_1(s,n)
 \sum_{t>s}g(s,t)t^{-2},
\]
which proves \eqref{eq:kernel-comparison} by
\eqref{eq:summable-growth}.

For the pathwise assertion, use the coupling of
Lemma~\ref{lem:one-edge} whenever the two tracked vectors agree.
After a mismatch evolve each vector according to its own transition
law. The comparison transitions have product laws conditional on
their past, so their coordinates remain independent. At a possible
first mismatch, the sum of tracked weights is at most $W_e$, and
the expected mismatch bound in \eqref{eq:one-edge-comparison} is at
most $C_h\sum_jg(b_{v_j},t)^2/t^2$ by the second moment estimate.
A union bound gives \eqref{eq:finite-coupling}. For
$v_j\in[\varepsilon k,Lk]$ and $t\le Tk$, $g(b_{v_j},t)$ is bounded
uniformly, and the sum in~\eqref{eq:finite-coupling} is at most
$C_{h,\varepsilon,L,T}/k$.
\end{proof}

\subsection{Limits of individual degrees}
The comparison bound gives the marginal limits used in
Section~\ref{sec:clt}, including the chain started with one additional
degree in~\eqref{eq:weighted-identity}.

\begin{lemma}\label{lem:degree-limit}
Let $s_n\to\infty$ with $s_n/n\to0$, and let a one-degree chain
$Y$ start at a fixed integer $d\ge m$ at time $s_n$. Then
\[
 \frac{Y(n)}{R_1(s_n,n)}\convd G_{d+\delta}\quad\hbox{in }\R.
\]
In particular, if $v_n\to\infty$ and $v_n/n\to0$, then
$D_{v_n}(n)/R_1(v_n,n)\convd G_r$ in $\R$.
\end{lemma}

\begin{proof}
For sufficiently large $s_n$, a chain started at $d$ has admissible
transition probabilities. Apply~\eqref{eq:kernel-comparison} to
$\varphi(j)=f\{j/R_1(s_n,n)\}$, for any bounded Lipschitz $f$.
The difference of the two expectations is at most
$C\Lip(f)(d+1)^2/s_n$. The comparison variable is
$B^{(\delta,d)}(\Theta(s_n,n))/R_1(s_n,n)$.
Lemma~\ref{lem:clock} gives $\Theta(s_n,n)\to\infty$, and
Lemma~\ref{lem:birth} identifies its limit as $G_{d+\delta}$.
The deterministic shift $\delta/R_1(s_n,n)$ tends to zero.
The assertion for $D_{v_n}$ follows by taking $d=m$.
\end{proof}

\section{Variance values and the stationary degree profile}
\label{sec:examples}
\subsection{Exact variance values}
For positive integers $a,b$, the beta representation gives exact
values without numerical quadrature, because
\[
 I_{1/2}(a,b)=2^{-(a+b-1)}
             \sum_{j=a}^{a+b-1}\binom{a+b-1}{j}
\]
by the binomial representation of the beta distribution function.
Substitution in~\eqref{eq:variance-beta} gives the following values.
\[
\begin{array}{cc|cc|c|c}
 m&\delta&p&r&I_{1/2}(r+p,r)&V_m(\delta)\\ \hline
 1&0&2&1&1/8&1/8\\
 2&0&2&2&3/16&7/64\\
 2&2&3&4&11/64&31/576\\
 3&0&2&3&29/128&19/192
\end{array}
\]
For a negative offset, take $m=2$ and $\delta=-1$. Then $r=1$ and
$p=3/2$. The minimum of two independent rate-one exponential variables
has rate two, so $\E(G\wedge G')^p/\E G^p=2^{-p}$ and
\[
 V_2(-1)=\frac49-\frac1{3\sqrt2}.
\]

For $m=1$ and $\delta=0$, $\alpha^2=1/4$ and $V_1(0)=1/8$.
The two contributions to the reduction are equal and together halve
the iid asymptotic variance.

\subsection{The stationary degree profile}
The mean degree profile also displays a first-order lattice effect.
For an integer $d\to\infty$, the stationary profile
in~\eqref{eq:pi} satisfies
\begin{equation}\label{eq:stationary-bias}
 \frac{\sum_{j>d}\pi_j\log(j/d)}{\sum_{j>d}\pi_j}
 =\alpha+\frac{B_{m,\delta}}d+O(d^{-2}),\qquad
 B_{m,\delta}=\frac{\delta+p+1/2}{p+1}.
\end{equation}
To verify the coefficient, the Gamma-ratio expansion
\citep[\S5.11(iii)]{dlmf} gives
\[
 \sum_{j>d}\pi_j
 =\beta^pd^{-p}\left\{1-\frac{p(2\delta+p+1)}{2d}
                              +O(d^{-2})\right\}.
\]
Also
$\pi_j=p\beta^pj^{-p-1}
\{1-(p+1)(2\delta+p)/(2j)+O(j^{-2})\}$.
The Euler--Maclaurin formula gives a summation error $O(d^{-p-2})$
for $x^{-p-1}\log(x/d)$ on $[d,\infty)$: its value at $d$ is zero,
its first derivative there is $d^{-p-2}$, and the integral of the
absolute value of its second derivative is $O(d^{-p-2})$.
For $x^{-p-2}\log(x/d)$ the error is $O(d^{-p-3})$.
The remaining terms contribute $O(d^{-p-2})$ after summation.
Integrating the two leading terms therefore gives
\[
 \sum_{j>d}\pi_j\log(j/d)
 =\beta^pd^{-p}\left\{\frac1p
                  -\frac{p(2\delta+p)}{2(p+1)d}
                  +O(d^{-2})\right\}.
\]
Division proves~\eqref{eq:stationary-bias}.
For a real threshold $a$, write $d=\lfloor a\rfloor$.
The conditioning event is unchanged, and
$\log(j/a)=\log(j/d)+\log(d/a)$.
Thus the first-order coefficient becomes
$B_{m,\delta}-\{a\}$, where $\{a\}=a-\lfloor a\rfloor$.
It can vanish or be negative for admissible parameters.
For example, $B_{m,\delta}=0$ at
$\delta=-5m/\{2(m+1)\}$ when $m\ge2$.
These cancellations are another reason not to assert universal
necessity of the rate from the bias calculation alone.

\setlength{\bibsep}{2pt}

\end{document}